\documentclass{amsart}
\usepackage{listings}
\usepackage{graphicx}
\newtheorem{theorem}{Theorem}[section]
\newtheorem{thm}[theorem]{Theorem}
\newtheorem{propo}[theorem]{Proposition}
\newtheorem{lemma}[theorem]{Lemma}
\newtheorem*{defn}{Definition}

\newtheorem{coro}[theorem]{Corollary}
\def\B#1#2{{#1\choose #2}}
\def\A{\mathcal{A}}
\def\B{\mathcal{B}}
\def\G{\mathcal{G}}
\def\F{\mathcal{F}}
\def\H{\mathcal{H}}
\def\N{\mathcal{N}}
\def\O{\mathcal{O}}

\def\U{\mathcal{U}}
\def\V{\mathcal{V}}
\def\W{\mathcal{W}}
\def\Id{{\rm Id}}
\def\dim{{\rm dim}}
\def\ran{{\rm ran}}
\def\ker{{\rm ker}}
\def\sign{{\rm sign}}
\def\tr{{\rm tr}}
\def\E{\rm E}

\title{A Brouwer fixed point theorem for graph endomorphisms}
\date{June 4, 2012}
\author{Oliver Knill}
\email{knill@math.harvard.edu}
\address{
        Department of Mathematics \\
        Harvard University \\
        Cambridge, MA, 02138\\
        }


\begin{document}
\maketitle

\begin{abstract}
We prove a Lefschetz formula $L(T) = \sum_{x \in \F} i_T(x)$ 
for graph endomorphisms $T: G \to G$, where $G$
is a general finite simple graph and $\F$ is the set of simplices fixed by $T$.
The degree $i_T(x)$ of $T$ at the simplex $x$ is defined as 
$(-1)^{\dim(x)} \sign(T|x)$, a graded
sign of the permutation of $T$ restricted to the simplex. The Lefschetz number $L(T)$ is 
defined similarly as in the continuum as $L(T) = \sum_k (-1)^k \tr(T_k)$,
where $T_k$ is the map induced on the $k$'th cohomology group $H^k(G)$ of $G$. 
A special case is the identity map $T$, where the formula reduces to the Euler-Poincar\'e formula 
equating the Euler characteristic with the cohomological Euler characteristic. 
The theorem assures that if $L(T)$ is nonzero, then $T$ has a fixed clique. A special
case is the discrete Brouwer fixed point theorem for graphs: if $T$ is a graph endomorphism of a 
connected graph $G$, which is star-shaped in the sense that only the zero'th cohomology group 
is nontrivial, like for connected trees or triangularizations of star shaped Euclidean domains, 
then there is clique $x$ which is fixed by $T$. If $\A$ is the automorphism group of a graph, 
we look at the average Lefschetz number $L(G)$. We prove that this is the Euler characteristic 
of the graph $G/\A$ and especially an integer. We also show that as 
a consequence of the Lefschetz formula, the zeta function 
$\zeta_T(z) = \exp(\sum_{n=1}^{\infty} L(T^n) \frac{z^n}{n})$ is a product
of two dynamical zeta functions and therefore has an analytic continuation as a rational function.
This explicitly computable product formula involves the dimension and the signature of prime orbits. 
\end{abstract}

\section{Introduction}

Brouwer's fixed point theorem assures that any continuous transformation 
on the closed ball in Euclidean space has a fixed point. First tackled by 
Poincar\'e in 1887 and by Bohl in 1904 \cite{Bohl} in the context of differential equations,
\cite{Zeidler}, 
%Bohl deals with differential equations and the connection is hard to see in the paper.
then by Hadamard in 1910 
%\footnote{\cite{brouwer1911} cites Hadamard in a book of J. Tannery}
and Brouwer in 1912 \cite{brouwer1911} in general, it is now a basic application in 
algebraic topology \cite{Hatcher,Mil65,Guillemin}. It has its use for example in game 
theory: the Kakutani generalization \cite{kakutani} has been used to prove 
Nash equilibria \cite{Franklin}. It is also useful for
the theorem of Perron-Frobenius in linear algebra \cite{Guillemin} which is one of the
mathematical foundations for the page rank used to measure the relevance of nodes in a network. 
More general than Brouwer is Lefschetz' fixed point theorem $\sum_{x \in F} i_T(x) = L(T)$ 
\cite{lefschetz49} from 1926 which assures that if the Lefschetz number $L(T)$ of a continuous transformation 
on a manifold is nonzero, then $T$ has a fixed point. In 1928, Hopf \cite{Hopf28} extended this to arbitrary finite
Euclidean simplicial complexes and prove that if $T$ has no fixed point then $L(T)=0$.
The third chapter of \cite{Dieudonne1989} and \cite{GD} provides more history. Brouwer's theorem follows from Lefschetz 
because a manifold $M$ homeomorphic to the unit ball has $H^k(M)$ are trivial for $k>0$ so that
$L(T)=1$ assuring the existence of a fixed point. \\

Since Brouwer's fixed point theorem has been approached graph theoretically with 
hexagonal lattices \cite{gale} or using results on graph colorings like the Sperner lemma 
\cite{Knaster}, it is natural to inquire for a direct combinatorial analogue on graphs without relating 
to any Euclidean structure. But already the most simple examples like rotation on a triangle show, 
that an automorphism of a graph does not need to have a fixed vertex, even if the graph is a 
triangularization of the unit disc. Indeed, many graph endomorphisms in contractible graphs 
do not have fixed points. 
Even the one-dimensional Brouwer fixed point theorem which is equivalent to 
the intermediate value theorem does not hold: a reflection $(a,b) \to (b,a)$ on a two point graph does
not have a fixed vertex.  \\

The reason for the failure is that searching for fixed vertices is too narrow. 
We do not have to look for fixed points but fixed simplices. These fundamental entities are also 
called cliques in graph theory. This is natural since already the discrete exterior algebra 
deals with functions on the set $\G = \bigcup_{k} \G_k$ of simplices of the graph $G$, where 
$\G_0=V$ is the set of vertices $\G_1=E$ 
is the number edges, $\G_2$ the set of triangles in $G$ etc. The Euler characteristic is the
graded cardinality $\sum_{k} (-1)^k |\G_k|=\sum_k (-1)^k v_k$ of 
$\G$. The role of tensors in the continuum is played by functions on $\G$. 
A $k$-form in particular is an antisymmetric function on $\G_k$. The definition of the 
exterior derivative $df(x)=\sum_i (-1)^i f(x_0,\dots,\hat{x_i},\dots,x_k) = f(\delta x)$ 
is already Stokes theorem in its core because for a $k$-simplex $x$, the boundary
$\delta x = \bigcup (x_0,\dots,\hat{x_i},\dots,x_k)$ is the union of $(k-1)$-dimensional simplices in $x$
which form the boundary of $x$. The definition of exterior derivative tells that $df$ evaluated
at a point $x$ is the same than $f$ evaluated at the boundary point $\delta x$. 
We see that in graph theory, the term "point" comes naturally when used for cliques of the graph. \\

Because of the atomic nature of cliques, we therefore prove a Lefschetz formula which holds 
for graph endomorphisms of finite simple graphs and where the simplices are the basic "points". 
Despite the discrete setting, the structure of the proof 
is close to Hopf's approach to the classical Lefschetz theorem \cite{Hopf28}. 
More text book proofs can now be found in \cite{JM,Spanier,GD}. %\cite{Brown}
While the definition of Lefschetz number goes over directly, it was the definition of the degree =index 
of a map at a fixed simplex which we needed to look for. Direct discretisations of the classical definition of the 
Kronecker-Brower degree $i_T(x) = \sign \det(1-dT(x))$ do not work. 
% Kronecker: papers 1869-1878, Brower 1912
%\footnote{Unlike the index $i_f(x)$ for Poincar\'e-Hopf \cite{poincarehopf}, 
%where the computer was essential, the computer implementation came here after thought experiments found the definition.} 
We found that the degree (which we also often call index)
$$  i_T(x) = (-1)^{\dim(x)} \sign(T|x) $$ 
leads to a theorem. In this definition $n=\dim(x)$ is the dimension
of the complete graph $x=K_{n+1}$ and $\sign(T|x)$ is the signature of the permutation induced on $x$. 
With this definition, every cyclic permutation on a simplex has index $1$ and the sum of the indices over all fixed 
subsimplices of a simplex is $1$ for any permutation $T$. This matches that $L(T)=1$ for any automorphism of 
a simplex $K_n$. 
The main result is the Lefschetz formula
$$ L(T) = \sum_{x \in \F(T)} i_T(x) \; , $$
where $\F(T)$ is the subset of $\G$ which is fixed by $T$. 
The proof uses the Euler-Poincar\'e formula which is the special case when $T$ is the identity.
A second part is to verify that for a fixed point free map $L(T)=0$. A final ingredient is show 
$L(f|\U \cup \V) = L(f|\U) +L(f|\V)$ 
for two $T$-invariant simplex sets $\U,\V$. 
The Lefschetz number applied to the fixed point set
is the Euler characteristic and equal to the sum of indices, the Lefschetz number applied to 
$\H$ is zero.  \\

The probabilistic link \cite{indexexpectation} between Poincar\'e-Hopf $\sum_{x \in V} i_f(x)=\chi(G)$ 
and Gauss-Bonnet $\sum_{x \in V} K(x) = \chi(G)$ obtained by integrating over all injective functions 
$f$ on the vertex set $V$ motivates to look for an analogue of Gauss-Bonnet in this context. 
This is possible: 
define a Lefschetz curvature $\kappa(x)$ on the set of simplices $x$ of $G$ as the rational number
$$  \kappa(x)= \frac{1}{|{\rm Aut}(G)|} \sum_{T \in {\rm Aut}_x(G)} i_T(x)  \; . $$
It is an almost immediate consequence of the Lefschetz formula that the Gauss-Bonnet type formula
\begin{equation}
 \sum_{x \in \G} \kappa(x) = L(G) \; 
\label{gaussbonnet}
\end{equation}
holds, where $L(G)$ is the average over $\A={\rm Aut}(G)$ of all automorphisms.
It is a graph invariant which refers to the symmetry group of the graph. 
Unlike the Euler characteristic it can be nonzero for odd dimensional
graphs. For one-dimensional geometric graphs for example, $L(G)$ is the number of connected components
and the curvature $\kappa(x)$ on each edge or vertex of  $C_n$ is constant $1/(2n)$.
For complete graphs, the Lefschetz curvature 
is concentrated on the set $\G_0$ of vertices and constant $1/(n+1)$.
An other extreme case is when $\A$ is trivial, where curvature
is $1$ for even-dimensional simplices and $-1$ for odd-dimensional simplices. The Gauss-Bonnet type 
Formula~(\ref{gaussbonnet}) is then just a reformulation of the Euler-Poincar\'e formula 
because $L(G)$ is then the cohomological Euler characteristic.  \\

While
$L(G)$ behaves more like a spectral invariant of the graph as the later also depends crucially on symmetries,
we will see that $L(G)$ is the Euler characteristic of
the quotient graph $G/\A$ of the graph by its automorphism group. The quotient graph is a discrete analogue
of an orbifold. In the case $P_3$ for example, where we have $6$ automorphisms, the Lefschetz numbers are
$(3, 1, 1, 0, 0, 1)$ with the identity $L(\Id)=\chi(G)=3$, the rotations $L(T)=0$ and reflections of two vertices
give $L(T)=1$. The average $L(G)=1$ is the Euler characteristic of $K_1 \sim G/\A$. For
the complete graph $K_3$, which has the same $6$ automorphisms, the Lefschetz numbers are $(1,1,1,1,1,1)$
and again $L(G)=1$. The proof that $L(G)$ is an Euler characteristic only uses the Burnside lemma
in group theory and is much simpler than the analogous result for orbifolds.  \\
%The invariant $L(G)$ is motivated by the Erlanger program of Klein who suggested to 
%classify a geometry $G$ by the group $\A$ of symmetries. 
%Averaging over all possible symmetries is always natural.  \\

Since the Lefschetz number is a weighted count of fixed points, the Lefschetz number of the iterate $T^n$ 
is a weighted count of periodic orbits. The Lefschetz zeta function 
$$ \zeta_T(z) = \exp(\sum_{n=1}^{\infty} L(T^n) \frac{z^n}{n})  \;  $$
encodes this. It is an algebraic version of the Artin-Mazur 
\cite{ArtinMazur} zeta function which in dynamical systems theory is
studied frequently \cite{ruellezeta}. Actually, as we will see in this article, the Lefschetz formula 
implies that the Lefschetz zeta function of a graph automorphism is the product of 
two zeta functions defined in dynamical systems theory. 
It therefore has a product formula. This formula is a finite product over all possible prime periods 
$$ \zeta_T(z) = \prod_{p=1}^{\infty} (1-z^p)^{a(p)-b(p)} (1+z^p)^{c(p)-d(p)} \; , $$
where $a(p)$ rsp. $c(p)$ is the number of odd-dimensional prime periodic orbits $\{x,Tx,\dots,T^{p-1}x\}$
for which $T^p|x$ has positive rsp negative signature and $b(p)$ rsp. $d(p)$ are
the number of even-dimensional prime periodic orbits for which $T^p|x$ has positive rsp. negative signature. \\

The zeta function (or rather its analytic continuation given by the rational function in the form of
the just mentioned product formula) contains the Lefschetz numbers of iterates of 
the map because it defines a Taylor series with periodic coefficients
$$  \frac{d}{dz} \log(\zeta(z)) = \sum_{n=1}^{\infty} L(T^n) z^{n-1} \; . $$
% compare partition function in statistical mechanics, where the right hand side is -energy
For the zeta function of a reflection at a circular graph $C_n$ for example,
where $\zeta(z)=(z+1)/(1-z)$ the right hand side is $2+2z^2+2z^4+\cdots$ 
showing that $L(T^n)=2$ for odd $n$. 
% f[z_]:=(z+1)/(1-z); g[z_]:=D[Log[f[z]],z]; Series[g[z],{z,0,10}]
An immediate consequence of product formulas for dynamical zeta functions gives a product formula
which is in the case of the identity $\zeta_{\Id}(z)=(1-z)^{-\chi(G)}$ again just a reformulation of
the Euler-Poincar\'e formula. \\
% A={{1/2,1/3},{1/2,2/3}}; f[z_]:=1/Det[1-A z^(-1)]; Series[f[z],{z,1,3}] 
As in number theory, where the coefficients in Dirichlet $L$-series are multiplicative characters, also
dynamical zeta function have coefficients which are multiplicative by definition. When using the
degree $i_T(x)$, this is not multiplicative because of the dimension factor $(-1)^{\dim(x)}$.
We can split the permutation part from the dimension part however and write a product 
formula for the zeta function which involves two dynamical zeta functions. \\

Because graphs have finite automorphism groups $\A={\rm Aut}(G)$, one can define a zeta function of the graph as
$$ \zeta_G(z) = \prod_{T \in Aut(G)} \zeta_T(z) \; . $$
As the Lefschetz zeta function of a transformation, the graph zeta function is a rational function.
For a reflection $T$ at a circular graph $C_n$
for example, we have $\zeta_T(z) = (1+z)/(1-z)$ because 
$L(T)=2,L(T^2)=0$ and $\exp(\sum_{n \; {\rm odd}} 2 z^n/n)=(1+z)/(1-z)$
so that $\zeta_G(z) = (1+z)/(1-z)$. For a graph with trivial automorphism group, 
we have $\zeta_G(z) = \zeta_{\Id}(z) = (1-z)^{-\chi(G)}$. These examples prompt the question about the
role of the order of the zero or of pole at $z=1$.
The order at $z=1$ is important wherever zeta functions appear, the original 
Riemann zeta function with a pole of order $1$.
%in algebraic geometry like the BSD conjectures it is related to algebraic properties of the objects
An other example is for subshifts of finite type with Markov matrix $A$, where 
the Bowen-Lanford formula \cite{BowenLanford} writes the dynamical zeta function 
as $\zeta(z) = 1/{\rm det}(1-A z)$ which by Perron-Frobenius has a pole of order 
$k$ at $z=1$ if $A$ has $k$ irreducible components. \\

The proof of the discrete Lefschetz formula is graph theoretical and especially does not involve
any limits. Like Sperner's lemma, it would have convinced the intuitionist Brouwer: 
%\footnote{Brouwer questioned his own fixed point theorem later in his life since it was not constructive} 
both sides of the Lefschetz formula can be computed in finitely many 
steps. The Lefschetz formula is also in this discrete incarnation a generalization
of the Euler-Poincar\'e formula which is a linear algebra result in the case of graphs. If we look at
the set $\G$ of all the simplices of a graph, then this set can be divided into a set $\F$ which 
is fixed and a set $\N$ which wander under the dynamics. The fixed simplices 
can be dealt with combinatorically.  \\

Lets see what happens in the special case when $T$ acts on a complete graph $G$ and
where $x$ is the simplex which is the entire graph. Understanding this is crucial.
The Lefschetz fixed point formula is $\sum_{y \in \F(T)} i_T(y) = L(T)=1$. 
Lets see the proof: the permutation $T$ induced on $G$ decomposes into cycles $y_1,\dots ,y_k$ 
which all are subsimplices of $G$. Since also arbitrary unions of simplices are simplices, the transformation 
$T$ fixes $2^k-1$ simplices which is the set of all subsets except the empty set 
which does not count as a fixed point.
We have $i_T(y_j)=1$ because the order of the cyclic permutation cancels the dimensional grading.
Next $i_T(y_i \cup y_j) = -1$ and in general 
$$   i_T(y)=(-1)^{|y|-1}  \; , $$
where $|y|$ is the number of orbits in $y$. In other words, $i_T(y)$ depends on 
the dimension of the "orbit simplex". 
Since $\sum_{y \in \F(T) \subset x} (-1)^{|y|-1} =0$ and $i(\emptyset)=-1$,
we have $\sum_{y \in \F(T) \subset x, y \neq \emptyset} i_T(y)=1$ which agrees with $L(T)=\chi(G) =1$. 

\section{The Lefschetz number}

Given a simple graph $G=(V,E)$ denote by $\G_k$ the set of complete $K_{k+1}$ subgraphs of $G$. 
Elements in $\G_k$ are called cliques. The set $\G_2$ is the set of all triangles 
in $G$, $\G_1=E$ the set of edges and $\G_0=V$ the set of vertices. If the cardinality 
of $\G_k$ is denoted by $v_k$, then the Euler characteristic of $G$ is defined as 
$$  \chi(G) = \sum_{k=0}^{\infty} (-1)^k v_k  \; , $$ 
a finite sum. \\

To get the discrete exterior bundle, 
define a $k$-form as a function on $\G_k$ which is antisymmetric in its $(k+1)$ arguments. 
The set $\Omega^k$ of all $k$-forms is a vector space of dimension $v_k$. 
The exterior derivative $d: \Omega_{k} \to \Omega_{k+1}$ is defined as
$df(x)=\sum_i (-1)^i f(x_0,\dots,\hat{x_i},\dots,x_k)$, where $\hat{x}$ denotes a variable taken away.
A form is closed if $df=0$. It is exact if $f=dg$. The vector space $H^k(T)$ of closed forms modulo exact forms 
is the cohomology group of dimension $b_k$, the Betti number. 
The cohomological Euler characteristic of $G$ is defined as 
$$  \sum_{k=0}^{\infty} (-1)^k b_k \; . $$
The sign ambiguity of forms can be fixed by defining an orientation on $G$. The later assigns a constant 
$n$-form $1$ to each maximal $n$-dimensional simplex; a simplex being maximal if it is not contained in a larger
simplex. $G$ is orientable if one can find an orientation which is compatible on the intersection of maximal
simplices. If $G$ should be nonorientable, we can look at a double cover $G'$ of $G$ and define $H^k(G)$ as $H^k(G')$.
A graph automorphism lifts to the cover and a fixed point in the cover projects down to a fixed point on $G$. \\

{\bf Example:} \\
Let $G$ be a triangle. The vector space of $0$-forms is three dimensional, the space of $1$-forms $3$-dimensional
and the space of $2$ forms one-dimensional. An orientation is given by defining $f(1,2,3)=1$ inducing orientations
on the edges $f(1,2)=f(2,3)=f(3,1)=1$. \\

A graph endomorphism is a map $T$ of $V$ such that if $(a,b) \in E$ then $T(a,b) \in E$. If $T$ is invertible, then
$f$ is called a graph automorphism. Denote by $T_p$ the induced map on the vector space 
$H^p(G)$. As a linear map it can be described by a matrix once a basis is introduced on $H^p(G)$. \\

{\bf Remark}. \\
We can focus on graph automorphisms, 
because the image ${\rm im}(T)$ is $T$-invariant and $T$ restricted to the attractor 
${\rm im}(T^n)$ for sufficiently large $n$ is an automorphism. This is already evident by ignoring the graph structure
when looking at permutations only.  \\

The following definition is similar as in the continuum %(but we take cohomology, not homology).

\begin{defn}
Given a graph endomorphism $T: G \to G$ on a simple graph $G$, define the Lefschetz number as 
$$ L(T) = \sum_{p=0}^{\infty} (-1)^p \tr(T_p)  \; ,  $$
where $T_k$ is the map $T$ induces on $H^k(G)$.
\end{defn}

%It is custom to attribute $Z_2$ graded sums with "super". The Euler characteristic
%$\chi(G)$ is the super cardinality of $G$, the number $L(T)$ is 
%the super trace of $T$ on cohomology. \\

{\bf Examples.}  \\
{\bf 1)} For the identity map $T$, the number $L(T)$ is the cohomological Euler characteristic of $G$. 
Denote by $\F(T)$ the set of {\bf fixed points} of $T$. In the same way as the classical Lefschetz-Hopf theorem 
we have then $L(T)= \sum_{x} i_T(x) = \sum_x (-1)^{\dim(x)}$, where $i_T$ is the index of the transformation. \\
{\bf 2)} If $G$ is a zero dimensional graph, a graph without edges, and $T$ is a permutation of $V$
then it is an automorphism and $L(T)$ is equal to the number of fixed points of $T$.  
The reason is that only $H^0(G)$ is nontrivial and has dimension $v_0$. The transformation $T_0$ is the permutation
defined by $T$ and $\tr(T_0)=v_0$.  \\
{\bf 3)} If $G$ is a complete graph, then any permutation is an automorphism. Only $H^0(G)$ is nontrivial and has 
dimension $1$ and $L(T)=1$.  \\
{\bf 4)} The tetractys is a graph of order $10$ which is obtained by dividing a triangle into 9 triangles.
The automorphism group is the symmetry group $D_3$ of the triangle. Again, since only $H^0(G)$
is nontrivial, $L(T)=1$ for all automorphisms. For rotations, there is only one fixed point, the central triangle. 
For reflections, we have $2$ vertices, $2$ edges and $3$ triangles fixed.  \\
{\bf 5)} For a cyclic graph $C_n$ with $n \geq 4$, both $H^0(G)$ and $H^1(G)$ are nontrivial. 
If $T$ preserves orientation and is not the identity, then there are no fixed points. The Lefschetz number is $0$. 
For the reflection $T$, the Lefschetz number is $2$. Any reflection has either $2$ edges or two vertices
or a vertex and an edge fixed. \\
{\bf 6)} The Petersen graph has order 10 and size 15 has Euler characteristic $-5$ 
and an automorphism group of 120 elements. The Lefschetz
number of the identity is $-5$, there are 24 automorphisms with $L(T)=0$ and $80$ automorphisms with $L(T)=1$ 
and 15 automorphisms with $L(T)=3$. The sum of all Lefschetz numbers is $120$ and the average Lefschetz number
therefore $1$. We will call this $L(G)$ and see that it is $\chi(G/\A)$, where $G/\A$ is the quotient graph
which consists of one point only. 

\section{Lefschetz fixed point theorem}

\begin{defn}
Denote by $\F(T)$ the set of simplices $x$ which are invariant under the endomorphism $T$. 
A simplex is invariant if $T(x)=x$. In this case, $T|x$ is a permutation of the simplex.
\end{defn}

\begin{defn}
For a fixed simplex $x$ in the graph $G$ and an endomorphism $T$, define the {\bf index}
$$  i_T(x) = (-1)^{\dim(x)} \sign(T|x) \; ,  $$
where $\sign(T|x)$ is the signature of the permutation $T$ induces on $x$.
The integer $\sign(T) \in \{-1,1 \; \}$ is the determinant of the corresponding permutation matrix. 
\end{defn}

{\bf Remarks}. \\
{\bf 1)} In the continuum, the inner structure of a fixed point is accessible through the derivative and
classically, $i_T(x) = \sign({\rm det}(dT(x)-I))$ is the index of a fixed point.
{\bf 2)} Is there a formal relation between the continuum and the discrete?  
In the continuum, we have $i_T(x)=p(1)$ where $p$ is the characteristic polynomial of $dT(x)$. 
In the discrete we have $i_T(x) = -p(0)$ where $p$ is the characteristic polynomial of
the permutation matrix $-P$ of $T$ restricted to $x$.  \\

{\bf Examples.} \\
{\bf 1)} If $x$ is a $0$-dimensional simplex (a vertex), 
then $i_T(x)=1$ for every fixed point $x$ and the sum of indices agrees with $L(T)$. \\
{\bf 2)} If $x$ is a $1$-dimensional simplex $K_2$ (an edge) and $f$ is the identity, then $i_T(x)=-1$. 
If $f$ flips two point in $x=K_2$, then $i_T(x)=1$.  \\
{\bf 3)} Let $G$ be a cyclic graph $C_n$ with $n \geq 4$. An automorphism is either a rotation or a
reflection. We have $\tr(T_1)=-1$ in the orientation preserving case and $\tr(T_1)=1$ in the 
orientation reversing case. For any invariant simplex, we have $i_T(x)=1$. \\
{\bf 4)} If $G$ is a wheel graph and $T$ is a rotation, then there is one fixed point and $L(T)=1$. The index of 
the fixed point is $1$ as any $0$-dimensional fixed point has.  \\
%{\bf 5)} For a reflection at the tetractys, there are 7 fixed points, 2 vertices of index $1$, $2$ edges of index $1$
%and 3 triangles of index $-1$. The sum of the indices is $1$. 

\begin{thm}[Lefschetz formula]
For any graph endomorphism $T$ of a simple graph $G$ with fixed point set $\F(T)$  we have
$$ L(T) = \sum_{x \in \F(T)} i_T(x) \; . $$
\label{lefschetz formula} 
\end{thm}

{\bf Examples.} \\
{\bf 1)} For the identity map $T(x)=x$ we have $L(T) = \chi(G)$ as in the continuum. 
The formula rephrases the Euler-Poincar\'e formula telling that the homological Euler 
characteristic is the graph theoretical Euler characteristic because 
$i_T(x)=(-1)^{\dim(x)}$ and every $x$ is a fixed point.  \\
{\bf 2)} Assume $G$ is a $1$-dimensional circular graph $C_n$ with $n \geq 4$. 
If $T$ is orientation preserving, then the Lefschetz number is $0$, otherwise it is 
$2$ and we have two fixed points. Lets compute $L(T)$ in the orientation preserving case: 
the space $H^0(G)$ is $R$ and the map $T$ induces the identity on it.
The space $H^1(G)$ consists of all constant functions on edges and $T$ induces $-\Id$.
If $T$ is orientation reversing, then the left hand side is $1-(-1)=2$. Indeed we  have then 
two fixed points.  \\
{\bf 3)} Assume $G$ is an octahedron and $T$ is an orientation preserving automorphims of $G$, 
then $T_0$ on $H^0(G)$ and $T_2$ on $H^2(G)$ are both the identity and since $H^1(G)$ 
is trivial, the Lefschetz number is $2$. There are always at least two fixed simplices. 
It is possible to have two triangles or two vertices invariant. \\
{\bf 4)} The Lefschetz number of any map induced on the wheel graphs is $1$ because only 
$H^0(G)$ is nontrivial. Any endomorphism has at least one fixed point. \\
{\bf 5)} If  $G$ is an icosahedron, then there are automorphims which have just two 
triangles fixed. Also two fixed points are possible. \\
{\bf 6)} Assume $T$ is an orientation reversing map, a reflection on an octahedron. 
We do not need to have a fixed point. Indeed, the map $T$ induced on $H^2(G)$ is $-1$
and the Lefschetz number is $L(G)=1-1=0$. \\
{\bf 7)} If $G$ consists of two triangles glued together at one edge, then $\chi(G)=1$. Take $T$
which exchanges the two triangles. This leaves the central edge invariant. The Lefschetz number
of $T$ is $1$. \\
{\bf 8)} For a complete graph $G=K_{n+1}$, any permutation is a graph automorphism. The Lefschetz
number is $1$ because only $H^0(G)$ is nontrivial. The index of any cyclic subsimplex
is $1$. As in the identity case, we have $\sum_{x \in \F(T)} i_T(x) = 1$, which is the $L(T)$.
As mentioned in the introduction one can see this special case as a Euler-Poincar\'e formula
for the orbit graph because $T$ on every cyclic orbit $y$ is a cyclic permutation with $i_T(y)=1$.  \\

The classical Poincar\'e lemma in Euclidean space assures that for a region homeomorphic to a 
star-shaped region only $H^0(G)$ is nonzero. This motivates to define:

\begin{defn}
A graph $G=(V,E)$ is called star-shaped if all vector spaces $H^k(G)$ are trivial for $k \geq 1$. 
\end{defn}

{\bf Examples.} \\
{\bf 1)} Given an arbitrary graph $H=(V,E)$, then the pyramid construction 
$$ G = (V \cup \{p\},E \cup \{ (v,p) \; | \; v \in V \; \} $$ 
is star-shaped. \\
{\bf 2)} Any tree $G$ is star-shaped as there are no triangles in the tree making $H^k(G)$ trivially
vanish for $k \geq 2$. The vector space $H^1(G)$ is trivial because there are no loops. \\
{\bf 3)} The complete graph is star-shaped. \\
{\bf 4)} The cycle graph $C_n$ is not star-shaped for $n>3$. \\
{\bf 5)} Any finite simply connected and connected subgraph of an infinite hexagonal graph is star-shaped. \\
{\bf 6)} The icosahedron and octahedron are both not star-shaped. Actually, any orientable 
         $2$-dimensional geometric graph (a graph where each unit sphere is a $1$-dimensional circular graph) is
         not star-shaped as Poincar\'e duality $H^0(G) \sim H^2(G)$ holds for such graphs. \\

As in the continuum, the Brouwer fixed point theorem follows:

\begin{theorem}[Brouwer fixed point]
A graph endomorphism $T$ on a connected star-shaped graph $G$ has a fixed clique. 
\end{theorem}

\begin{proof}
We have $L(T)=1$ because only $H^0(G)=R$ is nontrivial and $G$ is connected. 
Apply the Lefschetz fixed point theorem.
\end{proof}

\section{Proof}

We can restrict ourself to graph automorphisms because an endomorphism $T$ 
restricted to the attractor $G'=\bigcap_{k=0}^{\infty} T^k(G)$ of $T$ is an automorphism. Any
fixed point of $T$ is obviously in the attractor $G'$ so that the sum in the Lefschetz formula does not change
when looking at $T$ on $G'$ instead of $T$ on $G$. Also the Lefschetz number $L(T)$ does not change as
any invariant cohomology class, an
eigenvector $w$ of the linear map $L_k$ on the vector space $H^k(G)$ must be supported on $G'$. \\

The set of $\G$ is the union of the set $\F$ of simplices which are fixed and
the set $\N$ of simplices which are not fixed by the automorphism $T$.
It is possible that some elements in $\N$ can be a subsimplex of an element in 
$\F$. For a cyclic rotation on the triangle $K_3$ for example, the triangle 
itself is in $\F$ but its vertices are in $\N$.  \\
To see the Lefschetz number more clearly, we extend $T$ to $\G_k$.
Given a $k$-simplex $x$, it has an orbit $x,T(x),T^2(x),\dots,T_k^n(x)$ 
which will eventually circle in a loop since $T_k$ is a map on a finite set. 

\begin{defn}
The Euler characteristic of a subset $\A$ of $\G$ is defined as
$$  \chi(\A) = \sum_{p=0}^{\infty} (-1)^p |\A \cap \G_p| \; . $$
The Lefschetz number $L$ of $T$ restricted to an invariant set $\A$ is defined as
$$ L(T|\A) = \sum_{p=0}^{\infty} (-1)^p \tr(T_p| \A)  \; ,  $$
where $T_p$ is the map induced on the linear subspace generated by functions on $\A$.
\end{defn}

{\bf Remarks.} \\
{\bf 1)} The linear subspace generated by functions on $\A$ is in general not invariant under the
exterior derivative $d$: a function supported on $\A$ has $df$ which is
defined on $\G$ and not on $\A$ in general. \\
{\bf 2)} We have $\chi(\G)=\chi(G)$, where the left hand side is the Euler characteristic
of the super graph and the right hand side the Euler characteristic of the graph. Note however
that there are subsets of $\G$ which are not graphs. For a triangle for example, we can look
at the set $\V=\G_1$ of edges and get the Euler characteristic $\chi(\V)=-3$. 

\begin{lemma}[Additivity of $L$]
Assume $\U,\V$ are disjoint subsets of $\G$ and
assume both are $T$-invariant. Then 
$$  L(T|\U \cup \V) = L(T|\U) + L(T|\V) \; . $$
Especially, for $T=\Id$,
$$  \chi(\U \cup \V) = \chi(\U) + \chi(\V) \; . $$
\label{lemma1}
\end{lemma}

{\bf Remark.} \\
The vertex sets defined by two subsets $\U,\V \subset \G$ 
do not need to be disjoint. For a triangle $G=K_3$ for example, $\U=\G_3$ and 
$\V=\G_1$ are disjoint sets in $\G$ even so 
every $y \in \V$ is a subgraph of every nonempty subset of $\U$. 

\begin{proof}
We only have to show that
$$   L_k(\G) = L_k(\U) + L_k(\V) \;  .  $$
Let $\U_l$ be the set of $l$-simplices in $\U$,
$\V_l$ the set of simplices in $\V$.
The sets $\U_l$ are $T$ 
invariant for $l \leq k$. Any member $f_l$ of a cohomology class $H^l(G)$ is a 
function on $\U_l$ can be decomposed
as $f_l = f_{\U}$ where $f_{\U}$ has support in $\U_l$ etc. 
The matrix $L_k$ is a block matrix and the trace is the sum of the traces of the blocks. 
\end{proof}

\begin{coro}
Assume $\U,\V$ are $T$-invariant subsets of $\G$, then 
$$  L(T|\U \cup \V) = L(T|\U) + L(T|\V) - L(T|\U \cap V) \; . $$
\label{lemma2}
\end{coro}

\begin{proof}
Write $\W=\U \cap \V$ and apply lemma(\ref{lemma1}) twice for the disjoint sets 
$\U = \U \setminus \W,\W$ and then $\U,\V\setminus \W$ which has the 
union $\U \cup \V$. 
\end{proof}

%{\bf Remark.} \\
%this is related to the fundamental Mayer-Vietoris which assures the exactnes of
%$H^p(\U \cup \V) \to^I H^p(\U) \oplus H^p(\V) \to^J H^p(\U \cap \V) \to^d H^{p+1}(\U \cup \V)$,
%where $I(T) \to (f|\U,f|\V), J(f|\U,g|\V) = (f-g)|(\U \cap \V)$ and $\U,\V$ are $T$ invariant 
%subsets of $\G$. 

\begin{lemma}
If $T$ has no fixed point in $\G$, then $L(T)=0$. More generally $L(T|\N)=0$
if $T$ has no fixed points on $\N \subset \G$. 
\label{lemma3}
\end{lemma}

\begin{proof} 
Given a simplex $x$, the orbit $\U = \{ T^k(x) \; \}$ is
$T$-invariant and $L(T|\U)=0$. To see this, note that $H^k(\U)$ is trivial for $k\geq 2$. 
There are two possibilities: either $\U$ is connected, or $\U$ has $n$ connectivity components. 
In the first case, the orbit graph $\U$ has a retraction to a
cyclic subgraph so that $H^0(\U)=R$ and $H^1(\U)=R$. In that case, $T_0,T_1$ are both identities on 
$H^0(\U),H^1(\U)$ and $L(T)=0$. In the second case, $H^0(G)$ is $n$-dimensional and the only cohomology which is nontrivial.
The map $T$ is a cyclic permutation matrix on $H^0$ which has trace zero.
For two $T$ invariant sets $\U,\V \subset \G$, the intersection is also invariant and also 
has zero Lefschetz number. Therefore, the Lefschetz number of the union of all orbits is 
zero by lemma~(\ref{lemma1}). 
\end{proof}

The next lemma assures that for any finite simple graph $G$, 
the graph theoretical Euler characteristic is equal to the cohomological Euler 
characteristic. 

\begin{lemma}[Euler-Poincar\'e formula]
If $T$ is the identity, then $L(T) = \chi(G)$.
\label{lemma4}
\end{lemma}

In other words, for any simple graph, the cohomological Euler characteristic $L(\Id)$
is the same than the graph theoretical Euler characteristic.

\begin{proof}
This is linear algebra \cite{Eckmann1,Edelsbrunner}:
Denote by $C_m$ the vector space of $m$-forms on $G$. It has dimension $v_m$. 
The kernel $Z_m= {\rm ker}(d)$ of dimension $z_m$ and range $R_m={\rm ran}(d)$ of dimension $r_m$ 
From the rank-nullety theorem in linear algebra 
$\dim(\ker(d_m)) + \dim(\ran(d_m))=v_m$, we get 
\begin{equation}
 z_m = v_m-r_m  \; .  
\label{1}
\end{equation}
From the definition of the cohomology groups $H^m(G)=Z_m(G)/R_{m-1}(G)$ of dimension $b_m$ we get
\begin{equation}
 b_m = z_m-r_{m-1} \; . 
\label{2}
\end{equation}
Adding Equations (\ref{1}) and (\ref{2}) gives
$$ v_m-b_m = r_{m-1}+r_m  \; . $$
Summing this up over $m$ (using $r_{-1}=0, r_m=0$ for $m>m_0$) 
$$ \sum_{m=0}^{\infty} (-1)^m (v_m-b_m) = \sum_{m=0}^{\infty} (-1)^m (r_{m-1}+r_m) = 0 \;  $$
which implies $\sum_{m=0}^{\infty} (-1)^m v_m = \sum_{m=0}^{\infty} (-1)^m b_m$. 
\end{proof}

\begin{lemma}[Fixed point]
If $\F$ is the set of simplices in $G$ fixed by $T$ then 
$$   L(T|\F)= \chi(\F) = \sum_{x \in \F} i_T(x) \; . $$
\label{lemma4}
\end{lemma}

\begin{proof}
Because every $x \in \F$ is fixed
we have a disjoint union $\F = \bigcup_{x \in \F(T)} x$ and
$L(T) = \chi(\F)$.
Because $T$ is the identity on each fixed point, we have $i_T(x)=(-1)^{\dim(x)}$ and 
the second equality holds. 
\end{proof}

{\bf Examples}. \\
{\bf 1)} If there are $n$ maximal invariant simplices which do not intersect, then $L(T)=n$. This follows
from the additivity of $L$ and the fact that $T$ restricted to a simplex has $L(T|x)=1$ indepedent of $T$. \\
{\bf 2)} Let $G=K_{7}$. 
Let $T$ be a permutation with two cyclic orbits $y,z$ of order $3,4$ inside. The transformation has $3$
fixed points $x,y,z$ in total. We have 
$i_T(x)=(-1)^{3+4} \sign(T|x)=1$ and $i_T(y)=(-1)^3 \sign(T|y)=-1,i_T(y)=(-1)^4 \sign(T|y)=1$ and
the sum is $i_T(x)+i_T(y)+i_T(z)=1=L(T)$.  \\

Now to the proof of the theorem:  \\

\begin{proof}
The fixed point set $\F$ of $T$ is invariant and satisfies 
$L(T|\F) = \chi(\F) = \sum_{x \in \F(T)} i_T(x)$. 
The set $\N$ of simplices which are not fixed satisfies $L(T|\N)=0$. 
$$ L(T) = L(T|\G) = L(T|\F) + L(T|\N) 
        = L(T|\F) = \chi(\F) = \sum_{x \in \F(T)} i_T(x)  \; . $$
\end{proof}

{\bf Example.} \\
To illustrate the proof, look at an example, where we 
split a triangle into 4 triangles and rotate it by 60 degrees. 
The fixed point set $\F$ consists of the central triangle alone and the complement $\F^c = \H=\G \setminus \F$.
The fixed point set consists only of one point, the central triangle $x \in \G_2$. 
All other parts of the supergraph $\G$ move, including the edges and vertices of the triangle itself. \\

\section{Lefschetz curvature}

We have seen in \cite{indexexpectation} that averaging the Poincar\'e-Hopf 
index theorem formula \cite{poincarehopf}
$$ \sum_{x \in V} i_f(x) = \chi(G) \;  $$
over a probability space of all injective functions $f:V \to R$ leads to Gauss-Bonnet 
\cite{cherngaussbonnet}
$$ \sum_{x \in V} K(x) = \chi(G) \; . $$
It is therefore natural to look at the average $i_T(x)$ as a curvature 
when we sum up over all stabilizer elements in $\A_x = {\rm Aut}_x(G)$. 

\begin{defn}
Define the {\bf Lefschetz curvature} of a simplex $x \in \G$ as
$$  \kappa(x) = \frac{1}{|\A|} \sum_{T \in \A_x} i_T(x) $$
and the {\bf average Lefschetz number} 
$$  L(G) = \frac{1}{|\A|} \sum_{T \in \A} L(T)   $$ 
when averaging over all automorphisms. 
\end{defn} 

The number $L(G)$ has an interpretation as the expected index of fixed points a 
graph if we chose a random automorphism in the automorphism group. It is a lower
bound for the expected number of fixed points of a random automorphism on a graph. \\

{\bf Examples. } \\
{\bf 1) } For a cycle graph $C_n$ with $n \geq 4$, half of the automorphisms have 
$L(T)=0$ and half have $L(T)=2$. The average Lefschetz number is $1$. \\
{\bf 2) } For a complete graph $K_n$, all automorphisms satisfy $L(T)=1$ so that 
the average Lefschetz number is $1$. \\
{\bf 3) } For the Petersen graph $G$, the average Lefschetz number is $1$. \\

If the Lefschetz formula is compared with the Poincar\'e-Hopf formula, 
then $\kappa(x)$ is an analogue of Euler curvature and
the next result is an analogue of Gauss-Bonnet but we sum over all simplices in $G$. 
The Lefschetz curvature is a nonlocal property. It does not depend only on a small
neighborhood of the point but on the symmetries which fix the point = simplex. \\

\begin{thm}[Average Lefschetz]
$$ \sum_{x \in \G} \kappa(x)  = L(G) \; . $$
\label{averagelefschetz}
\end{thm}
\begin{proof} 
Use the Lefschetz fixed point theorem to sum over $\A={\rm Aut}(G)$:
\begin{eqnarray*}
   L(G) &=& \frac{1}{|\A|} \sum_{T \in \A} L(T)   \\
        &=& \frac{1}{|\A|} \sum_{T \in \A} \sum_{x \in \F(T)} i_x(T)  \\
        &=& \frac{1}{|\A|} \sum_{x \in \G} \sum_{T \in \A_x} i_x(T)  \\
        &=& \sum_{x \in \G}  \frac{1}{|\A|} \sum_{T \in \A_x} i_x(T) 
         =  \sum_{x \in \G} \kappa(x) \; . 
\end{eqnarray*}
\end{proof}

{\bf Remark}.  \\
Unlike the Gauss-Bonnet theorem, this Gauss-Bonnet type theorem sums over all possible
simplices $\G$, not only vertices $V$ of the graph. 
The Lefschetz curvature is constant on each orbit of the automorphism group $\A$
and the sum over all curvatures over such an equivalence class is an integer $1$ or $-1$. Theorem~(\ref{averagelefschetz})
is the Euler-Poincar\'e formula in disguise since we will interpret $L(G)$ as an Euler characteristic of an "orbifold" graph
which in the discrete is just a graph. \\

{\bf Examples.} \\
{\bf 1)} If $G$ is the complete graph $K_{n+1}$, then $\A=S_{n+1}$ is the full permutation group 
and since $L(T)=1$ for all $T$, we also have $L(G)=1$. Now lets compute the Lefschetz curvature.
For every fixed $x$ we have $i_T(x)=(-1)^{\dim(x)} \sign(T|x)$ 
and averaging over all $T$ gives zero except if $x$ is a vertex, where 
$$  \kappa(x) = \frac{|\A_x|}{|\A|} = \frac{1}{n+1} \; .  $$
The Lefschetz curvature of a vertex is the same than the Euler curvature of a vertex.
The curvature is zero on $\G_k$ for $k>0$ because the indices of even odd dimensional permutations cancel. \\
{\bf 2)} If $G$ is star shaped then $L(T)=1$ for all $T$ and $L(G)=1$. It reflects the Brouwer analogue that 
every transformation has a fixed point. If $G$ is a star graph $S_n$, then the automorphism group is $D_n$. 
For the center point $i_T(x) = 1$ for all transformations and $\kappa(x) = 1$. All other points have $\kappa(x)=0$. 
While the Euler curvature is positive at the spikes and negative in the center, the Lefschetz curvature is 
entirely concentrated at the center. \\
{\bf 3)} If $G = C_n$ for $n \geq 4$, then $\A=D_{2n}$ 
is the dihedral group.  For reflections we have $L(T)=2$, for the rotations, 
$L(T)=0$. Therefore, $L(G)=1$. The stabilizer group $\A_x(G)$
consists always of two elements whether it is a vertex or edge and $i_T(x)=1$ in both cases. We have 
$\kappa(x) = 1/(2n)$ and $\sum_x \kappa(x) = 1$. The curvature is located both on vertices and
edges. Unlike the Euler curvature, the Lefschetz curvature is now nonzero. \\
{\bf 4)} If $G$ is the wheel graph $W_n$ with $n \geq 4$, then again $\A$ is the dihedral group. We still
have $L(G)=1$ but now $L(T)=1$ for all automorphisms. The center vertex has the full automorphism group as stabilizer
group and $i_T(x)=1$ for any transformation. Therefore $\kappa(x)=1$ at the center and $\kappa(x)=0$ 
everywhere else. The center has grabbed all curvature. \\
{\bf 5)} If $G$ has a trivial automorphism group, then $L(G)=\chi(G)$ is the Euler characteristic. Also each 
stabilizer group is trivial and $i_T(x)=(-1)^{|x|}$ so that $\kappa(x) = i_T(x) = (-1)^{|x|}$. In this case, 
the curvature is spread on all simplices, even-dimensional ones have positive curvature and odd-dimensional
ones have negative curvature. It is amusing that the Euler-Poincar\'e formula can now be seen as a 
Gauss-Bonnet formula for Lefschetz curvature.  \\
{\bf 6)} For the octahedron $G$, the orientation preserving automorphisms $T$ satisfy $L(T)=2$. They are 
realized as rotations if the graph is embedded as a convex regular polygon. The orientation reversing 
automorphisms have $L(T)=0$. The average Lefschetz number is $L(G)=1$ and the Lefschetz curvature is constant
$1$ at every point.  \\
{\bf 7)} We can look at the Erdoes-R\'enyi probability space $\Omega_n$ \cite{erdoesrenyi59} of 
$2^n$ graphs $G$ on a vertex set with $n$ vertices. The number $L(G)$ is a random variable  on 
$\Omega_n$. We computed the expectation $\E_n[L]$ for small $n$ as follows: 
$\E_2[L]=1, \E_3[L]=11/8,\E_4[L]=43/32, \E_5[L]=1319/1024, \E_6[L]=8479/8192$. 
Like Euler characteristic expectation of random graphs, the 
expectation of $L(G)$ is expected to oscillate more and more as $n \to \infty$. 
While $L(G)$ takes values $1$ or $2$ in the case $n=1,\dots,5$, there are graphs on 
$6$ vertices, where the maximal Lefschetz number is $3$ and the minimal $0$. 
The computation for $n=6$ is already quite involved since we have 32768 graphs and 
look at all the automorphisms and for each automorphism find all fixed points. \\
% A={1,1,11/8,43/32,1319/1024,8479/8192}; ListPlot[A]

The average Lefschetz number $L(G)$ obtained by averaging over the automorphism group $\A={\rm Aut}(G)$ is
always an integer, since it is the Euler characteristic of agraph. 

\begin{defn}
Let $G/\A$ be the orbigraph defined by the automorphism group $\A$ acting on $G$.
Two vertices are identified if there is an automorphism mapping one into the other. 
\end{defn}

{\bf Remark.} \\
$G/\A$ is a graph if we assume equivalence classes of vertices are
connected, if some individual vertices had been connected. 
If geometric graphs $G$ in which unit spheres have topological properties
from spheres and fixed dimension are considered discrete analogues of 
manifolds and $\B$ is a subgroup of automorphisms of $G$, then then $G/\B$ 
plays the role of orbifolds. Examples are geometric graphs with boundary, 
where each unit sphere is either sphere like or a half sphere of the same
fixed dimension. 

\begin{thm}[Average Lefschetz is Euler characteristic]
The Lefschetz number satisfies $L(G) = \chi(G/\A)$ and is an integer. 
The sum of the Lefschetz curvatures in an equivalence class of simplices 
is either $1$ or $-1$. 
\end{thm}

\begin{proof}
The proof only uses elementary group theory and some combinatorics about the indices, as well
as Theorem~(\ref{lefschetz formula}). \\
1) First, the Burnside lemma for the finite group $\A$ acting on $\G$ 
$$ |\G/\A| = \frac{1}{|\A|}  \sum_{T \in \A} |\F^T| \; , $$ 
where $\F^T$ is the set of fixed points of $T$ and $\G/\A$ is the set of simplices in $G/\A$.  \\
2) The number $(-1)^{\dim(x)}$ of a simplex $x \in G/\A$ is equal to the index $i_T(y)$ for 
every simplex $y$ which projects onto $x$. 
Proof. We have seen in the introduction that the dimension $\dim(y) = |y|-1$ 
of the simplex of $T|y$ satisfies 
$$ i_T(y)=(-1)^{|y|-1}  \; . $$
3) Let $\G_+$ be the set of simplices $x$ which are mapped under $\G \to \G/\A$ to an 
even dimensional simplex. These are the simplices $y$ for which $T|y$ have index $1$
independent of $T$. Similarly, let $\G_-$ be the set of simplices which are projected
to an odd dimensional simplex. All these simplices have negative index for all $T \in \A$. 
We therefore know that we have a partition $\G = \G_+ \cup \G_-$ and that for any $T \in \A$
and every $y \in \G$ the index $i_y(T)$ is equal to $(-1)^{\dim(x)}$ where $x=y/\A$.  \\
4) We can now use the Burnside lemma restricted to $\A$ invariant sets $\G_+,\G_-$ and get
$$ |\G/\A|_{2k}   = \frac{1}{|\A|} \sum_{T \in \A} |\F^T_+| \; , $$ 
$$ |\G/\A|_{2k+1} = \frac{1}{|\A|} \sum_{T \in \A} |\F^T_-| \; , $$ 
where $\F^T_{\pm}$ is the set of fixed simplices $y$ of $T$ for which $i_T(y)=\pm 1$.  \\
5) Let now $|\G/\A|_k$ the set of simplices in $\G/\A$ which have dimension $k$. 
We use the Lefschetz fixed point formula to finish the proof: 
\begin{eqnarray*}
   \chi(\G/\A) &=& \sum_{k=0}^{\infty} (-1)^k |\G/\A|_k 
                =  \frac{1}{|\A|} \sum_{T \in \A} |\F^T_+|-|\F^T_-| \\
               &=& \frac{1}{|\A|} \sum_{T \in \A} \sum_{x \in \F(T)} i_T(x) 
                =  \frac{1}{|\A|} \sum_{T \in \A} L(T)  = L(G)  \; . 
\end{eqnarray*}
\end{proof}

{\bf Remarks.} \\
{\bf 1)} Since $L(G)=\chi(G/\A)$ and $\kappa$ is constant on each orbit, the 
Lefschetz curvature of a simplex $x$ can be rewritten as $(-1)^{|x/\A|}/|\A x|$ where $x/\A$
is the simplex after identification with $\A$ and $\A x$ is the orbit of $x$ under
the automorphism group. Since $L(G) = \chi(G/\A)$ the Gauss-Bonnet type formula~(\ref{averagelefschetz}) is also
equivalent to an Euler-Poincar\'e formula in general. The number $\kappa(x)$ encodes so the orbit length of $x$
under the automorphism group $\A$.  \\
{\bf 2)} One can also see this as graded summation of an elementary result in linear algebra
(see \cite{HirzebruchZagier} page 21): if a finite group acts linearly on a finite 
dimensional vector space $V$, then $\dim(\F) = (1/|\A|) \sum_{T \in \A} \tr(T)$.
Proof. Let $j: \F \to V$ be the inclusion. Define $f(v) =  1/|\A| \sum_{T \in A} T(v)$.
The image of $f$ is in $\F$. If $\pi:V \to \F$ is the projection then $f=j \pi$.
If $v \in \F$, then $Tv=v$ for all $T \in A$ so that $f(v)=v$. Therefore $\pi j = \Id|\F$.
and $\dim(F) = \tr(\Id|\F) = \tr(\pi j) = \tr(j \pi) = \tr(f)$. If $\A$ is cyclic this simplifies to
$\dim(\F) = \tr(T)$.  \\

{\bf Examples.} \\
{\bf 1)} Let $G$ be the complete graph  $K_n$. Its automorphism group has $2^k$ elements. The orbifold graph is a single point. 
The average Lefschetz number is $1$. \\
{\bf 2)} Let $G$ be cycle graph $C_n$. The automorphism group is the dyadic group $D_n$ with $2n$ elements. The orbigraph
is again a single point. The average Lefschetz number is $1$. \\
{\bf 3)} Let $G$ be the discrete graph $P_n$. Its automorphism group is the full permutation group again. The orbifold graph
is a single point. The average Lefschetz number is $1$. \\
{\bf 4} Let $G$ be the octahedron. Its automorphism group has $48$ elements. The orbigraph is again a single point and
the average Lefschetz number is $1$. \\

{\bf Remark.} \\
The analogue statement for manifolds needs more algebraic topology 
like the Leray-Serre spectral sequence \cite{mathoverflow51993}: 
if a manifold $G$ has a finite group $\A$ of symmetries, then the average
Lefschetz number $L(T)$ of all the symmetry transformations $T$ is 
the Euler characteristic $\chi(O)$ of the orbifold $O=M/\A$.

\section{Lefschetz zeta function}

Having a Lefschetz number, it is custom to define a Lefschetz zeta function which encodes the 
Lefschetz numbers of iterates $T^n$ of the graph automorphism $T$. Zeta functions are one of 
those objects which are interesting in any mathematical field, whether it is number theory,
complex analysis, topology, dynamical systems or algebraic geometry. The case of 
graph theory considered here is a situation where one can see basic
ideas like analytic continuation work. For any pair $(G,T)$ where $T$ is an automorphism of a 
finite simple graph, we can construct an explicit rational function $\zeta(z)$. The product
formula we will derive allows to compute this function by hand for small graph dynamical systems. 

\begin{defn}
The Lefschetz zeta function of an automorphism $T$ of a graph $G$ is defined as
$$ \zeta_T(z) = \exp(\sum_{n=1}^{\infty} L(T^n) \frac{z^n}{n})  \; . $$
\end{defn}

For example, $T$ is a reflection of a circular graph $C_5$ where $L(T^2)=L(\Id)=\chi(G)=0$ and 
$L(T)=L(T^3) = L(T^5) = \dots = 2$, we have
$$  \zeta_T(z) =  \exp(\sum_{n=1}^{\infty} 2 \frac{z^{2n-1}}{2n-1}) = \exp( \log(1+z)-\log(1-z))  = \frac{1+z}{1-z}  \; . $$

The Lefschetz zeta function is an algebraic version of the Artin-Mazur zeta function 
\cite{ArtinMazur}. It is already interesting for the identity since 
$$ \zeta_{\Id}(z) = \exp(\sum_{n=1}^{\infty} L(T^n) \frac{z^n}{n})
            = \exp(\sum_{n=1}^{\infty} \chi(G) \frac{z^n}{n})
            = \exp(-\chi(G) \log(1-z) ) = (1-z)^{-\chi(G)}  \; . $$

\begin{propo}
$\zeta_T(z)$ is a rational function.
\end{propo}

\begin{proof}
By definition $L(T^n) = \sum_k (-1)^k \tr(T_k^n)$ we see that 
$\exp(\sum_n {\rm tr}(T_k^n) (-1)^n z^n/n)=\exp(-\log(1-z T_k)) = \det(1-z T_k)^{-1}$
for every $k$ and so 
$$  \zeta_T(z) = \prod_{k=1}^{\infty} {\rm det}(1-z T_k)^{(-1)^{k+1}} \; . $$
\end{proof}

The Lefschetz formula allows to write this as a product over periodic simplex orbits. Let $\F(T^n)$
denote the set of fixed simplices of $T^n$.  Then 
\begin{equation}
   \zeta_T(z) = \prod_{m=1}^{\infty} \prod_{x \in \F(T^m)} \exp( i_{T^m}(x) \frac{z^m}{m} ) \;  
              = \exp(\sum_{m=1}^{\infty} \frac{z^m}{m} \sum_{x \in \F(T^m)} i_{T^m}(x))  \; . 
\label{zetawithindex}
\end{equation}
Since $i_{T^n}(x) = (-1)^{\dim(x)} \prod_{k=0}^{m=1} \phi(T^kx)$, where 
$\phi(y) = {\rm det}(P_T(y))$ is the determinant of the permutation $y \to T(y)$ induced on the 
simplex, we can write this as 
$\zeta_{T|\E}(z)/\zeta_{T|\O}(z)$, 
a quotient of two dynamical systems zeta function  
$$ \zeta(z) = \exp(\sum_{m=1}^{\infty} \frac{z^m}{m} \sum_{x \in \F(T^m)} \prod_{k=0}^{m-1} \phi(T^k x)) $$
with $\phi \in \{-1,1 \}$ giving the sign of the permutation.  \\

\begin{defn}
Let $\F(p)$ be the set of periodic orbits of minimal period $p$. They are called prime orbits. 
Let  $a(p)$ rsp. $c(p)$ be the number of odd dimensional prime periodic orbits $\{x,Tx,\dots,T^{p-1}x \; \}$
for which $T^p|x$ has positive rsp negative signature.
Let $b(p)$ rsp. $d(p)$ be the number of odd-dimensional prime periodic orbits for which $T^p|x$
has positive (rsp.) negative signature.
\end{defn}

One only has to remember: "signature and $z$-sign  flip flop" and
"dimension has exponents 'odd on top'".

\begin{thm}[Product formula]
The zeta function of an automorphism $T$ on a simple graph $G$ is the rational function
$$ \zeta_T(z) = \prod_{p=1}^{\infty} (1-z^p)^{a(p)-b(p)} (1+z^p)^{c(p)-d(p)}  \; . $$
\end{thm}
\begin{proof}
Because prime periods are smaller or equal than the product of the cycle lengths of the permutation,
the product is finite. While we can follow the computation from the book \cite{ruellezeta} almost verbatim,
there is a twist: since the permutation part of the index
is multiplicative when iterating an orbit, the dimension part is not. If the dimension is odd, then
each transformation step changes the sign in the inner sum of the zeta function. 
If we write $i_T(x) = (-1)^{\dim(x)} \phi(x)$, where $\phi(x)$ is the sign of the permutation,
then only the $\phi$ part is multiplicative. Let $x$ be a periodic orbit of minimal period $p$. 
If we loop it $q$ times, we can write 
$\Psi(x^q) = \prod_{k=0}^{pq-1} \phi(T^kx) = [\prod_{k=0}^{p-1} \phi(T^kx)]^q = [\Psi(x)]^q$. 
As in \cite{ruellezeta}: 
$$ \sum_{x \in \F(T^m)} \Psi(x^m)  
 = \sum_{p|m} \sum_{x \in \F(p)} p \Psi(x^{m/p}) \; .$$
The definition (\ref{zetawithindex}) gives (substituting $q=m/p$ in the third identity):
\begin{eqnarray*}
   \zeta(z) &=& \exp(\sum_{m=1}^{\infty} \frac{z^m}{m}              \sum_{x \in \F(m)} \prod_{k=0}^{m-1} \phi(T^kx) ) 
             =  \exp(\sum_{m=1}^{\infty} \sum_{p|m} \frac{z^m}{m} p \sum_{x \in \F(p)} \Psi(x^{m/p})  ) \\
            &=& \exp(\sum_{p=1}^{\infty} \sum_{q=1}^{\infty} \frac{z^{pq}}{q} \sum_{x \in \F(p)} \Psi(x^q) ) 
             =  \exp(\sum_{p=1}^{\infty} \sum_{q=1}^{\infty} \sum_{x \in \F(p)} \frac{z^{pq}}{q} \Psi(x^q) ) \\
            &=& \exp(\sum_{p=1}^{\infty} \sum_{x \in \F(p)} \sum_{q=1}^{\infty} \frac{z^{pq}}{q} \Psi(x^q) )
             =  \exp(\sum_{p=1}^{\infty} \sum_{x \in \F(p)} \sum_{q=1}^{\infty} \frac{(z^p)^q}{q} [\Psi(x)]^q ) \\
            &=& \exp(\sum_{p=1}^{\infty} \sum_{x \in \F(p)} -\log(1-z^{p} \prod_{k=0}^{p-1} \phi(T^kx)) ) 
             =  \prod_{p=1}^{\infty} \prod_{x \in \F(p)} [(1-z^p \prod_{k=0}^{p-1} \phi(T^kx))]^{-1} \\
            &=& \prod_{x \in \F} [(1-z^{p(x)} \prod_{k=0}^{p(x)-1} \phi(T^kx))]^{-1} 
             =  (1-z^{p})^{a(p)-b(p)} (1+z^p)^{c(p)-d(p)} \; . 
\end{eqnarray*}
It is in the last identity that we have split up $\F$ into $4$ classes, depending on whether the dimension is 
even or odd or whether the permutation $T^{p(x)}$ on $x$ is even or odd. If the signature is $-1$, then this
produces an alternating sum before the log comes in which leads to a $(1+z^p)^{\pm 1}$ factor depending on the dimension.
If the signature is $1$, then we have $(1-z^p)^{\pm}$ factors depending on the dimension. \\
In the case $T=\Id$ for example, where the signature is always $1$, we have $(1-z)^{a-b}$ where $a$ is the 
number of odd fixed points and $b$ the number of even fixed points so that it is $(1-z)^{-\chi(G)}$. 
\end{proof} 

\begin{coro}
If $G$ is the union of two disjoint graphs $G_1,G_2$ and $T$ is an automorphism of $G$
inducing automorphisms $T_i$ on $G_i$, then
$$ \zeta_T(z) = \zeta_{T_1}(z) \zeta_{T_2}(z) \; . $$
\label{union1}
\end{coro}
\begin{proof}
The numbers $a(p),b(p),c(p),d(p)$ are additive. 
\end{proof}

{\bf Remarks.} \\
1) As in number theory, product formulas are typical for zeta
function. The prototype is the Euler product formula or so called golden key
$$ \zeta(s) = \prod_p (1-p^{-s})^{-1}   = \prod_p (1-z^{\log(p)})^{-1}  \; , $$
where $z =e^{-s}$ was plugged in  just to get formally more close to the dynamical formula above
and explain the etymology of the dynamical-zeta function. \\
2) One usually asks for a functional equation in the case of zeta functions.  If the number
of even dimensional and odd dimensional fixed points correspond, we have a symmetry $z \to -z$.  \\

{\bf Examples.} \\
{\bf 1)} For $T=\Id$, we have $a(1) = \sum_{k \; {\rm odd}} v_k$ and $b(1)=\sum_{k \; {\rm even}} v_k$
so that $\zeta_{\Id}(z) = (1-z)^{a-b} = (1-z)^{-\chi(G)}$.  \\
{\bf 2)} For a reflection $T$ at $C_4$, we have two periodic vertex orbit of period $p=1$, 
one periodic vertex orbit of period $2$ and two edge orbits of period $2$. The 
product formula gives 
$$  \frac{(1-z^2)^2}{(1-z)^2(1-z^2)} = \frac{1+z}{1-z} \; . $$
{\bf 3)} For a rotation $T$ of $C_4$ we have a periodic vertex orbit of period $4$ 
and a periodic edge orbits of period $4$. The product formula gives $1$. 
This follows also directly from the definition since $L(T^k)=0$ for all $k$.  \\
{\bf 4)} For any automorphism $T$ of the complete graph $G=K_n$, we have $L(T^n)=1$ so that 
$\zeta_T(z) = (1-z)^{-1}$.  \\
{\bf 5)} For the identity on the Petersen graph, we have $10$ fixed vertices of index $1$ 
and $15$ edges of index $-1$
$\zeta(\Id) = (1-z)^{15}/(1-z)^{10} = (1-z)^5$ reflecting the fact that $\chi(G)=-5$.   \\

Finite graphs have finite automorphism groups so that one can look at
$$ \zeta_G(z) = \prod_{T \in \A} \zeta_T(z) \; . $$

\begin{coro}
$\zeta(z)$ is a rational function.
\end{coro}
\begin{proof}
It is a finite product of rational functions. 
\end{proof} 

\begin{coro}
If $G$ is the union of two disjoint graphs $G_i$, then 
$$ \zeta_G(z) = \zeta_{G_1}(z) \zeta_{G_2}(z) \; . $$
\label{union2}
\end{coro}
\begin{proof}
This follows from Corollary~(\ref{union1}).
\end{proof}

{\bf Examples.} \\
{\bf 1)} We have seen $\zeta_{T}(z) = (1+z)/(1-z)$ for reflections and $\zeta(z)=1$ 
for rotations so that
$\zeta_{C_n}(z) =  (\frac{1+z}{1-z})^n$. \\
{\bf 2)} If $G$ has a trivial automorphism group, 
the product formula is equivalent to the Euler-Poincar\'e formula and
$\zeta_G(z) = \zeta_{\Id}(z) = (1-z)^{-\chi(G)}$. \\
{\bf 3)} For the complete graph $K_n$, we have 
$\zeta_G(z) = (1-z)^{-n!}$. 
The order of the pole at $z=1$ is the size of the automorphism group. \\
{\bf 4)} For the Petersen graph, we computed
$$ \zeta_G(z) = (1-z)^{10} (1+z)^{90} (1+z^2)^{30} (1+z+z^2)^{40} (1-z^4)^{30}(1-z^5)^{24} (1-z^6)^{20} \; . $$

{\bf Remark}. \\
There are other zeta functions for graphs. The Ihara zeta function
\cite{Terras} is defined as $\prod_{p} (1-u^{|p|})^{-1}$ where $p$ runs
over all closed prime paths in the graph and $|p|$ is its length. 
For $C_n$, it is $(1-z^n)^{-2}$ because
there are only two prime paths and both have length $n$. The Ihara zeta function appears unrelated to the 
above zeta function and is closer to the Selberg zeta function \cite{ruellezeta},
where the geodesic flow play the role of  automorphism. Both are of course isomorphism invariants. 
Unlike the average Lefschetz number $L(G)$ which is also an isomorphism invariant, the zeta 
function encodes more information about the graph than $G/\A$. 

\vspace{12pt}
\bibliographystyle{plain}
%\bibliography{geometry,edu}

\begin{thebibliography}{10}

\bibitem{ArtinMazur}
M.~Artin and B.~Mazur.
\newblock On periodic points.
\newblock {\em Ann. of Math. (2)}, 81:82--99, 1965.

\bibitem{Knaster}
S.~Mazurkiewicz B.~Knaster, C.~Kuratowski.
\newblock Ein {Beweis} des {Fixpunktsatzes f\"ur n-dimensional Simplexe}.
\newblock {\em Fundamenta Mathematicae}, 14:132--137, 1929.

\bibitem{Bohl}
P.~Bohl.
\newblock Ueber die {Bewegung} eines mechanischen {Systems} in der {N\"ahe}
  einer {Gleichgewichtslage}.
\newblock {\em J. Reine Angew. Math.}, 127:179--276, 1904.

\bibitem{BowenLanford}
R.~Bowen and O.~E. Lanford, III.
\newblock Zeta functions of restrictions of the shift transformation.
\newblock In {\em Global {A}nalysis ({P}roc. {S}ympos. {P}ure {M}ath., {V}ol.
  {XIV}, {B}erkeley, {C}alif., 1968)}, pages 43--49. Amer. Math. Soc., 1970.

\bibitem{brouwer1911}
L.E.J. Brouwer.
\newblock \"{U}ber {A}bbildung von {M}annigfaltigkeiten.
\newblock {\em Math. Ann.}, 71(1):97--115, 1911.

\bibitem{Dieudonne1989}
J.~Dieudonne.
\newblock {\em A History of Algebraic and Differential Topology, 1900-1960}.
\newblock Birkh\"auser, 1989.

\bibitem{mathoverflow51993}
J.~Ebert.
\newblock Euler characteristic of orbifolds.
\newblock
  http://mathoverflow.net/questions/51993/euler-characteristic-of-orbifolds,
  2011.

\bibitem{Eckmann1}
B.~Eckmann.
\newblock The {Euler} characteristic - a few highlights in its long history.
\newblock In {\em Mathematical Survey Lectures: 1943-2004}, 1999.

\bibitem{Edelsbrunner}
H.~Edelsbrunner and J.~Harer.
\newblock {\em Computational topology, An introduction}.
\newblock American Mathematical Society, Providence, RI, 2010.

\bibitem{erdoesrenyi59}
P.~Erd\"os and A.~R{\'e}nyi.
\newblock On random graphs. {I}.
\newblock {\em Publ. Math. Debrecen}, 6:290--297, 1959.

\bibitem{Franklin}
J.N. Franklin.
\newblock {\em Methods of Mathematical Economics, Linear and Nonlinear
  Programming, Fixed-Point Theorems}, volume~37 of {\em Classics in applied
  Mathematics}.
\newblock SIAM, 2002.

\bibitem{gale}
D.~Gale.
\newblock The game of {H}ex and the {B}rouwer fixed-point theorem.
\newblock {\em Amer. Math. Monthly}, 86(10):818--827, 1979.

\bibitem{GD}
A.~Granas and J.~Dugundji.
\newblock {\em Fixed Point Theory}.
\newblock Springer Monographs in Mathematics. Springer Verlag, 2003.

\bibitem{Guillemin}
V.~Guillemin and A.~Pollack.
\newblock {\em Differential topology}.
\newblock Prentice-Hall, Inc., New Jersey, 1974.

\bibitem{Hatcher}
A.~Hatcher.
\newblock {\em Algebraic Topology}.
\newblock Cambridge University Press, 2002.

\bibitem{HirzebruchZagier}
F.~Hirzebruch and D.~Zagier.
\newblock {\em The {A}tiyah-{S}inger theorem and elementary number theory}.
\newblock Publish or Perish Inc., Boston, Mass., 1974.
\newblock Mathematics Lecture Series, No. 3.

\bibitem{Hopf28}
H.~Hopf.
\newblock A new proof of the {Lefschetz} formula on invariant points.
\newblock {\em Proc. Nat. Acad. Sci.}, 14:149--153, 1928.

\bibitem{JM}
J.~Jezierski and W.~Marzantowicz.
\newblock {\em Homotopy methocs in topological fixed and periodic points
  theory}.
\newblock Springer Verlag.

\bibitem{cherngaussbonnet}
O.~Knill.
\newblock A graph theoretical {Gauss-Bonnet-Chern} theorem.
\newblock {\\}http://arxiv.org/abs/1111.5395, 2011.

\bibitem{poincarehopf}
O.~Knill.
\newblock A graph theoretical {Poincar\'e-Hopf} theorem.
\newblock {\\}http://arxiv.org/abs/1201.1162, 2012.

\bibitem{indexexpectation}
O.~Knill.
\newblock On index expectation and curvature for networks.
\newblock {\\}http://arxiv.org/abs/1202.4514, 2012.

\bibitem{lefschetz49}
S.~Lefschetz.
\newblock {\em Introduction to Topology}.
\newblock Princeton University Press, 1949.

\bibitem{Mil65}
J.~Milnor.
\newblock {\em Topology from the differential viewpoint}.
\newblock University of Virginia Press, Charlottesville, Va, 1965.

\bibitem{ruellezeta}
D.~Ruelle.
\newblock {\em Dynamical Zeta Functions for Piecewise Monotone Maps of the
  Interval}.
\newblock CRM Monograph Series. AMS, 1991.

\bibitem{kakutani}
S.Kakutani.
\newblock A generalization of {B}rouwer's fixed point theorem.
\newblock {\em Duke Math. J.}, 8:457--459, 1941.

\bibitem{Spanier}
E.H. Spanier.
\newblock {\em Algebraic Topology}.
\newblock Springer Verlag, 1966.

\bibitem{Terras}
A.~Terras.
\newblock {\em Zeta functions of Graphs}, volume 128 of {\em Cambridge studies
  in advanced mathematics}.
\newblock Cambridge University Press.

\bibitem{Zeidler}
E.~Zeidler.
\newblock {\em Nonlinear Functional Analysis and its Applications I}.
\newblock Springer, 1986.

\end{thebibliography}

\pagebreak

\begin{figure}
\scalebox{0.25}{\includegraphics{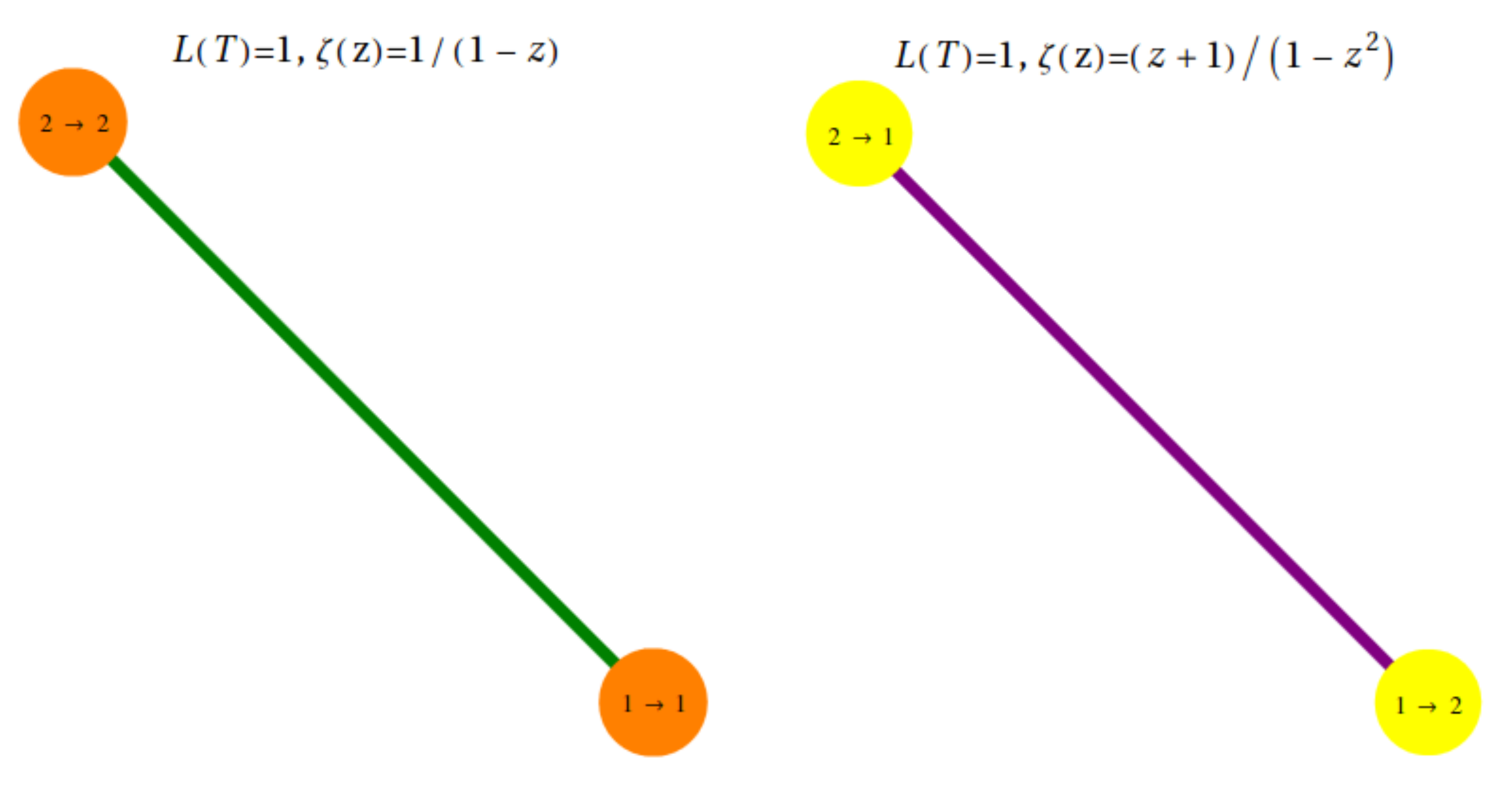}}
\caption{
The complete graph $G=K_2$ has $\A=S_2$ as automorphism group. All
transformations $T$ have Lefschetz number $1$. 
The zeta function of $T=\Id$ only involves
$a(1)=1,b(1)=2$ so that $(1-z)^{1-2}$.
The reflection $T$ has a fixed $K_2$ of negative signature giving $c(1)=1$ 
and a $0$-dimensional periodic point of period $2$ giving $b(2)=1$ 
so that $\zeta(z) = (1+z)/(1-z^2) = 1/(1-z)$. 
\label{example1}
}
\end{figure}

\begin{figure}
\scalebox{0.30}{\includegraphics{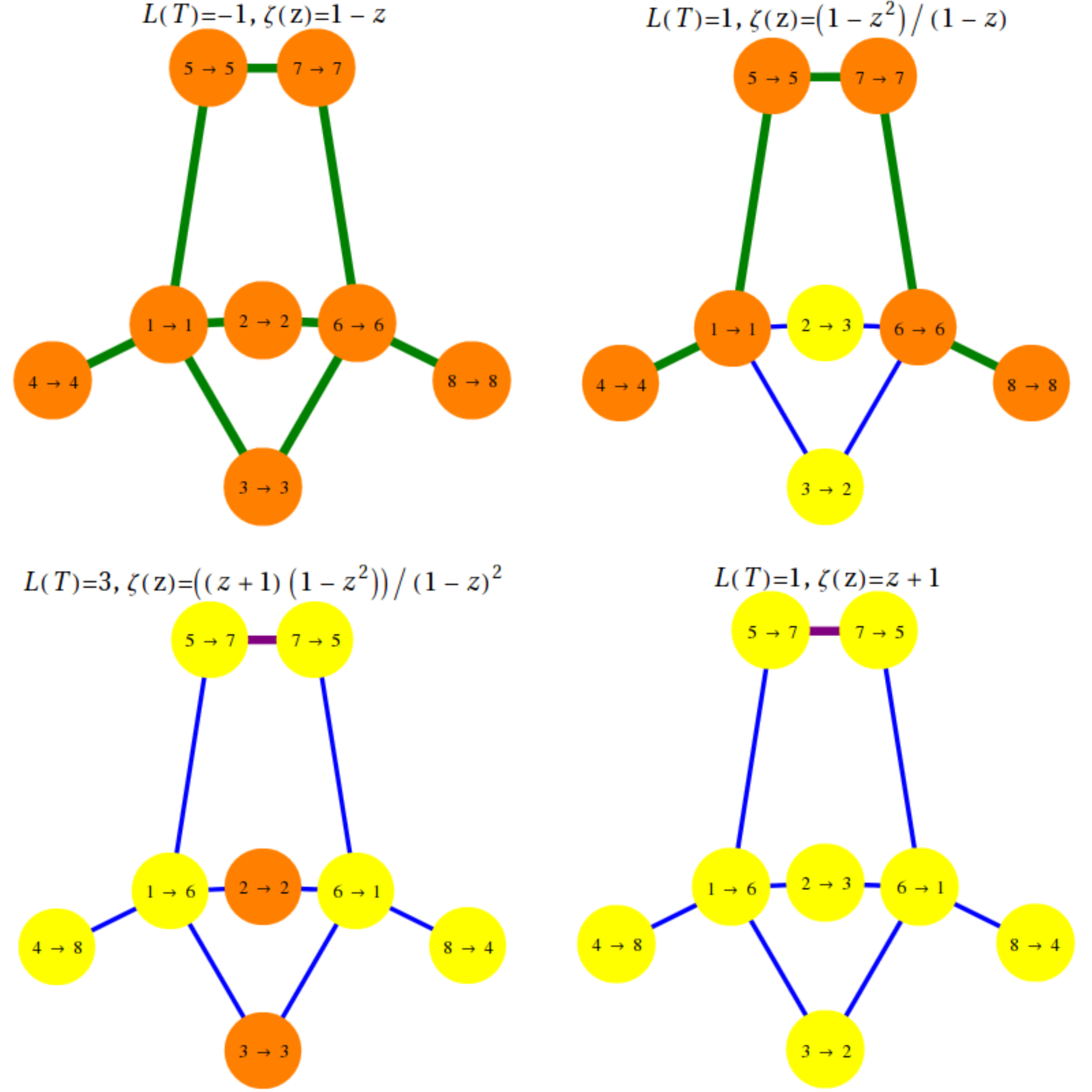}}
\caption{
A graph $G$ of order $8$ and size $9$ with automorphism group $S_2 \times S_2$. 
$T_1=\Id$  with $8$ fixed vertices of index $1$ 
and $9$ fixed edges of index $-1$ has Lefschetz number $L(T)=-1$. 
$T_2$ has $6$ fixed vertices of index $1$ and $5$ fixed edges of index $-1$ 
with Lefschetz number $L(T_2)=1$. 
$T_3$ has two fixed vertices of index $1$ and one edge of index $1$ 
leading to $L(T_3)=3$.  $T_4$ finally has only one fixed edge of index $1$. 
\label{example1}
}
\end{figure}

\begin{figure}
\scalebox{0.30}{\includegraphics{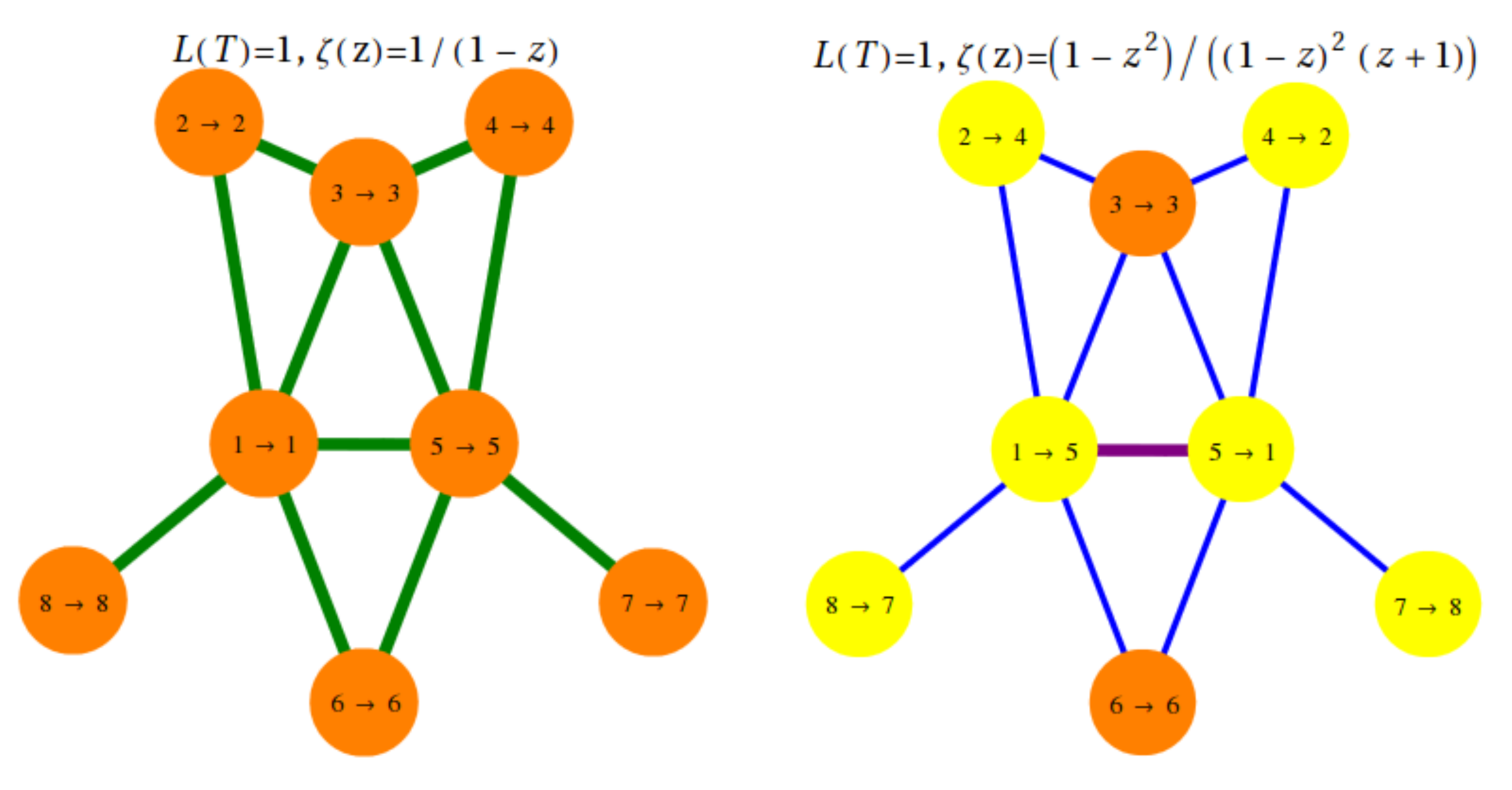}}
\caption{
A graph with a symmetry group of 2 elements. The reflection has
$2$ fixed vertices of index $1$, one fixed edge of index $1$ and $2$ fixed
triangles of index $1$.
\label{example4}
}
\end{figure}

\begin{figure}
\scalebox{0.25}{\includegraphics{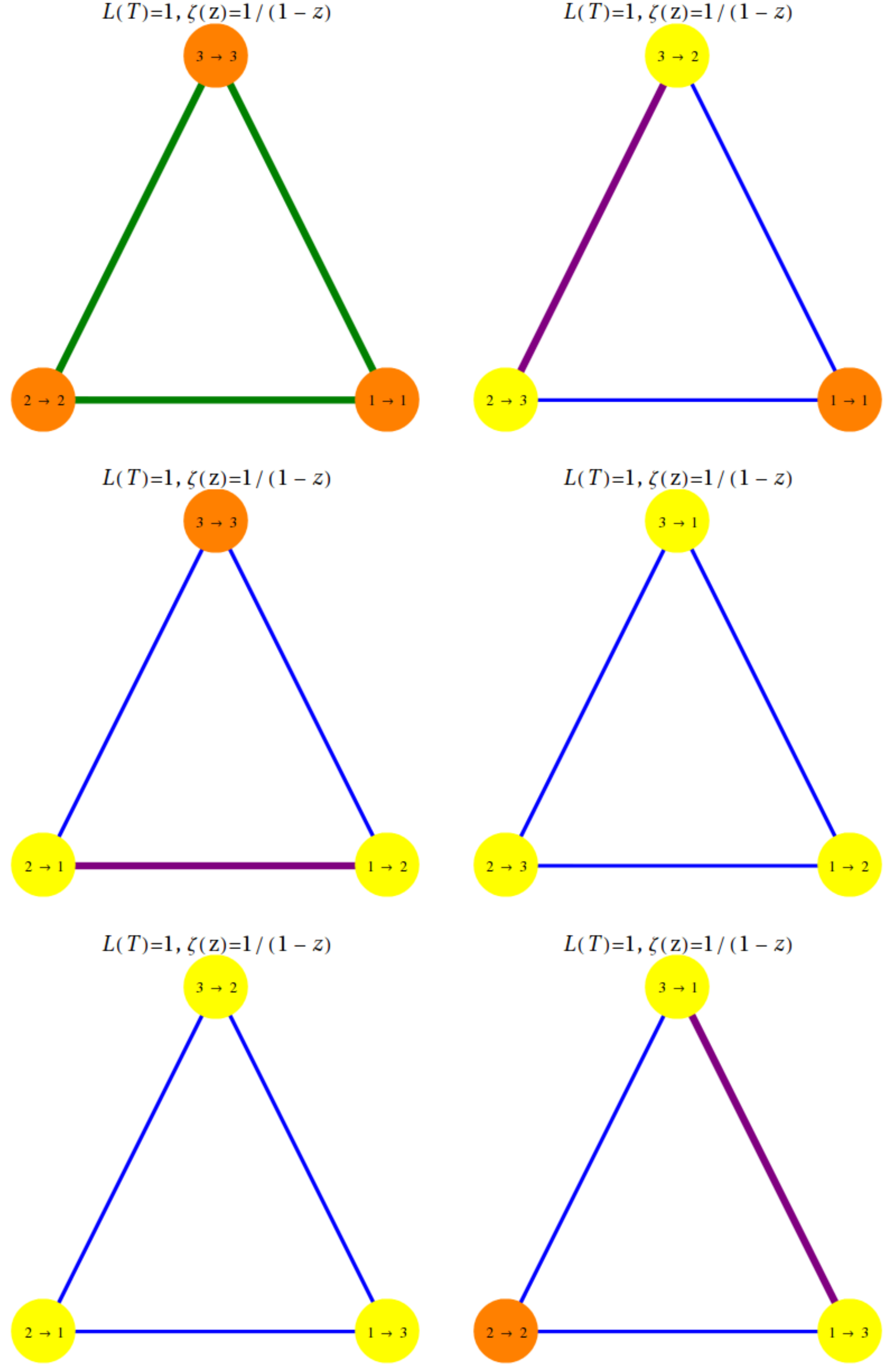}}
\caption{
The complete graph $G=K_3$ has $\A=S_3$ as automorphism group. All
transformations $T$ have Lefschetz number $L(T)=1$. 
\label{example3}
}
\end{figure}

\begin{figure}
\scalebox{0.20}{\includegraphics{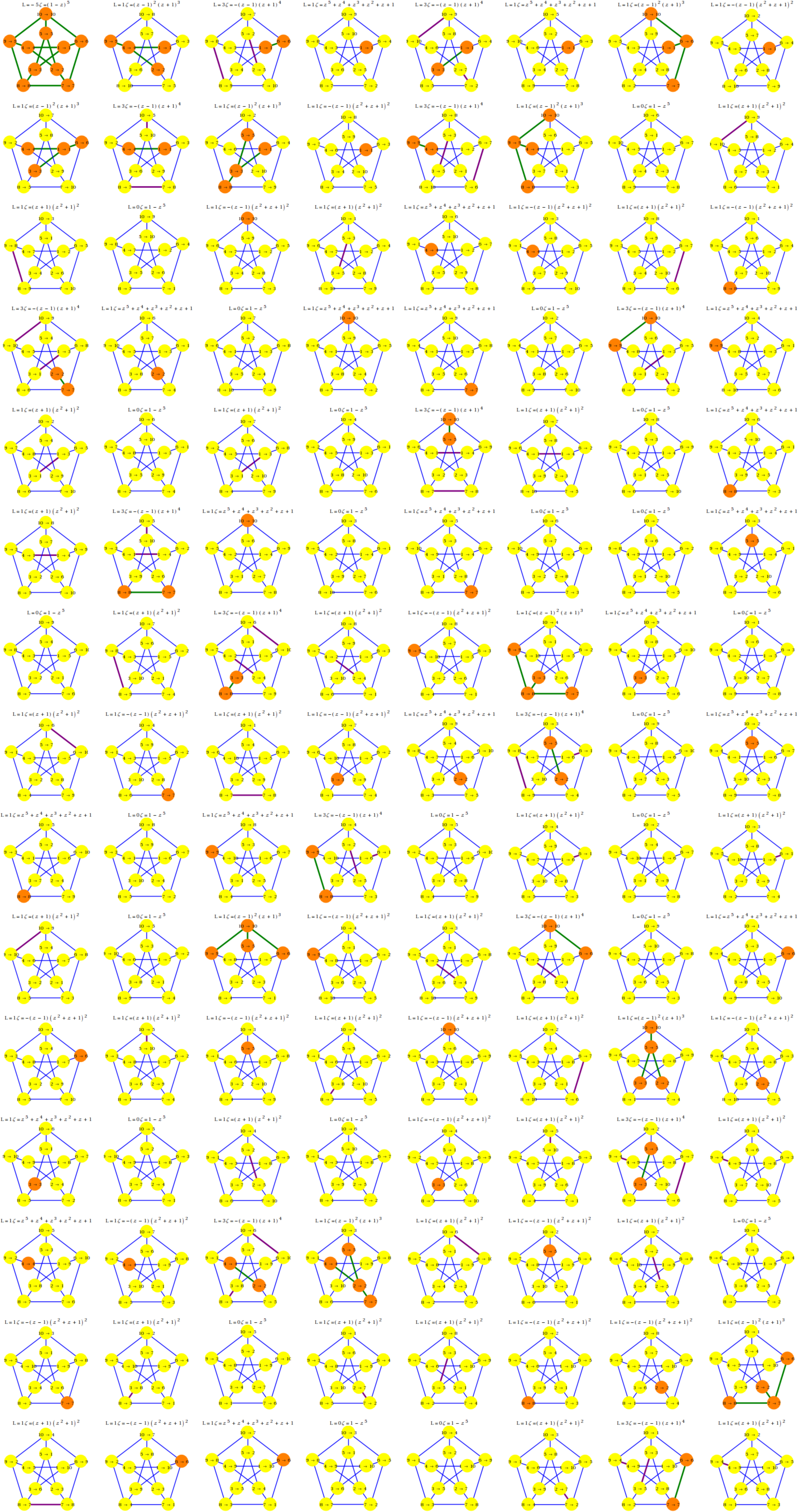}}
\caption{
The 120 automorphisms of the Petersen Graph. In each case, $L(T)$ and $\zeta(z)$ are computed and
the fixed vertices marked.
\label{example2}
}
\end{figure}

\begin{figure}
\scalebox{0.26}{\includegraphics{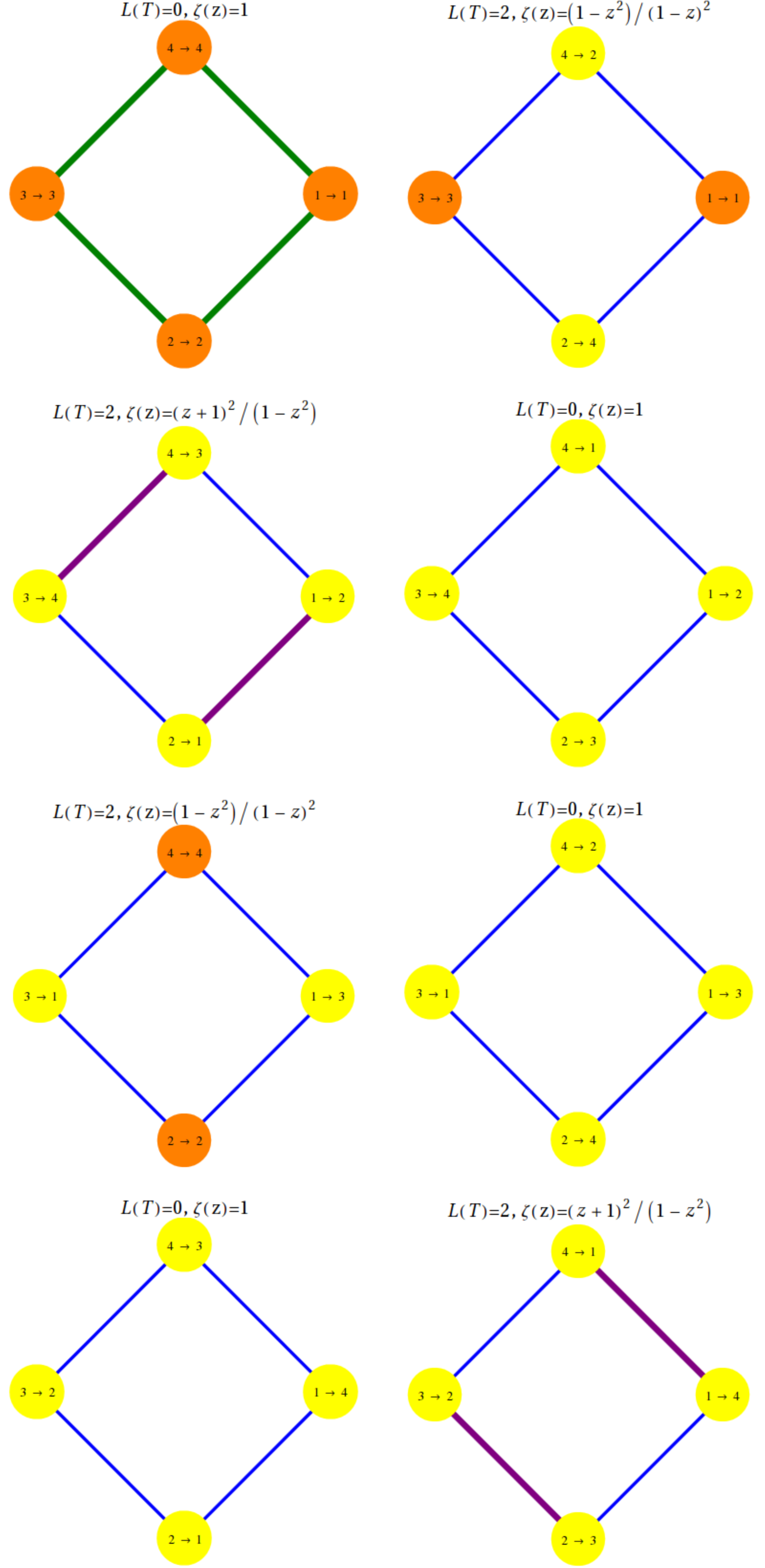}}
\caption{
The graph $G=C_4$ has as
the automorphism group of the graph $G=C_4$ is the dihedral group $D_4$.
The automorphism $T_1=\Id$ with $4$ fixed vertices of index $1$ and $4$ fixed edges of index $-1$ with $L(T_1)=0$.
There are 4 rotations which have no fixed points and $L(T)=0$.
There are 2 reflections which fix two vertices of index $1$.
There are 2 reflections which fix two edges of index $1$. 
All reflections have Lefschetz number $2$.
\label{example2}
}
\end{figure}

\end{document}